\documentclass{amsart}%
\usepackage{amsfonts}
\usepackage{amsmath}
\usepackage{amssymb}
\usepackage{graphicx}%
\newtheorem{theorem}{Theorem}[section]
\theoremstyle{plain}

\newtheorem{corollary}[theorem]{Corollary}

\newtheorem{definition}[theorem]{Definition}
\newtheorem{example}[theorem]{Example}

\newtheorem{lemma}[theorem]{Lemma}

\newtheorem{proposition}[theorem]{Proposition}
\newtheorem{remark}[theorem]{Remark}

\numberwithin{equation}{section}
\begin{document}
	\title[Poles of the Non-Archimedean Zeta Functions for Rational Functions]{Poles of Non-Archimedean Zeta Functions for Non-degenerate Rational Functions}
	\author{M. Bocardo-Gaspar}
	\address{CUCEI, Departamento de
		Matemáticas, Blvd. Gral. Marcelino García Barragán $\#$ 1421, Guadalajara, Jal., C.P. 44430, México}
	\email{miriam.bocardo@academicos.udg.mx}

	\thanks{}
	
	\begin{abstract}
		In this article, we study local zeta functions over non-Archimedean locals fields of arbitrary characteristic attached to rational functions and characters $\chi$ of the units of the ring of integers $\mathcal{O}_{K}$, by using an approach based on the multivariate $\pi$-adic stationary phase formula and Newton polyhedra. When the rational function is non-degenerate with respect to its Newton polyhedron, we give an explicit formula for the local zeta function and a list of the possible poles in terms of the normal vectors of the supporting hyperplanes of the Newton polyhedron attached to the rational function and their expected multiplicities. Furthermore, we obtain some conditions under which the local zeta function attached to the trivial character has at least one real pole by describing the largest negative real pole and the smallest positive one.
	\end{abstract}
	\maketitle

	\section{Introduction}

	Let $K$ be a non-Archimedean local
	field of arbitrary characteristic. By a well-known classification theorem, a
	non-Archimedean local field is a finite extension of $\mathbb{Q}_{p}$, the
	field of $p$-adic numbers, or the field of formal Laurent series
	$\mathbb{F}_{q}\left(  \left(  T\right)  \right)  $ over a finite field
	$\mathbb{F}_{q}$. For further details the reader may consult \cite[Chapter 1]{WeilB}. We denote by $\mathcal{O}_{K}$ the ring of integers of $K$ and let $\mathbb{F}_{q}$ be the residue field of
	$K$, the finite field with $q=p^{m}$ elements, where $p$
	is a prime number. For $z\in K\smallsetminus\left\{  0\right\}  $, let
	$\nu(z)\in\mathbb{Z}\cup\left\{  +\infty\right\}  $ denote \textit{the
		valuation} of $z$, let $|z|_{K}=q^{-\nu(z)}$ denote the normalized
	\textit{absolute value} (or \textit{norm}), \ and let $ac(z)=z\pi^{-\nu(z)}$
	denote the \textit{angular component}, where $\pi$ is a fixed uniformizing
	parameter of $K$. We extend the norm $|\cdot|_{K}$ to $K^{n}$ by taking
	$||(x_{1},\ldots,x_{n})||_{K}:=\max\left\{  |x_{1}|_{K},\ldots,|x_{n}%
	|_{K}\right\}  $. Then $(K^{n},||\cdot||_{K})$ is a complete metric space and
	the metric topology is equal to the product topology.

	Let $\mathcal{O}_{K}^{\times}$ be the multiplicative group of
	$\mathcal{O}_{K}$. A \textit{character} of $\mathcal{O}_{K}^{\times}$ is a
	continuous homomorphism $\chi:\mathcal{O}_{K}^{\times}\rightarrow{\mathbb{S}%
	}^{1}$ where ${\mathbb{S}}^{1}$ is the circle in $\mathbb{C}$ considered as a
	multiplicative group. We will call the characters of $\mathcal{O}_{K}^{\times
	}$ \textit{the multiplicative characters of} $K$, because given any character
	$\chi$ of $\mathcal{O}_{K}^{\times}$, the mapping $x\rightarrow\chi\left(
	ac\left(  x\right)  \right)  $ gives rise to a character of the multiplicative
	group $K^{\times}$ of $K$. We formally put $\chi(0)=0$. We will denote by $\chi_{_{triv}}$ the trivial
	character of $\mathcal{O}_{K}^{\times}$. Since
	$\mathcal{O}_{K}^{\times}$ is a totally disconnected, compact, abelian group,
	there exists $e\in$ $\mathbb{N}$ such that $\chi|_{1+\pi^{e}\mathcal{O}_{K}%
	}=1$. The smallest positive integer satisfying this condition $e=e_{\chi}$ is
	called the \textit{conductor of the character} $\chi$. Notice that $e_{\chi_{triv}}=1$ and 
	$e_{\chi}=e_{\chi^{-1}}$.
	
	Let $f,g\in \mathcal{O}_{K}[x_{1},\ldots, x_{n}]\backslash\pi\mathcal{O}_{K}[x_{1},\ldots, x_{n}]$ be two non-constant co-prime polynomials, $n\geq 2$.
	The \textit{twisted local zeta function associated to a rational function $f/g$ and a character $\chi$} is defined as 
	\begin{equation}
	Z\left(s,\chi,\frac{f}{g}\right)=\int\limits_{{\mathcal{O}}_{K}^{n}\backslash D_{K}}%
	\chi\left(  ac\left(  \frac{f(x)}{g(x)}\right)  \right)  \left\vert
	\frac{f(x)}{g(x)}\right\vert ^{s}|d\boldsymbol{x}|; \label{Integral_1}%
	\end{equation}
	where $D_{K}=\{x\in\mathcal{O}_{K}:f(x)=0\}\cup \{x\in\mathcal{O}_{K}:g(x)=0\}$ and $|d\boldsymbol{x}|_{K}$ denotes
	the Haar measure on $\left(  K^{n},+\right)  $, normalized such that the measure
	of $\mathcal{O}_{K}^{n}$ is equal to one. If $\chi=\chi_{triv}$ we use the notation $Z\left(s,\frac{f}{g}\right)$ instead of  $Z\left(s,\chi_{triv},\frac{f}{g}\right)$.
	
	These local zeta functions were introduced in \cite{V-Z2} by Veys and Z\'{u}\~{n}iga-Galindo. They extended Igusa's
	theory \cite{Igusa}, to the case of meromorphic
	functions $f/g$, with coefficients in a local field of characteristic zero. In  \cite{V-Z2}, by using resolution of singularities, it was stablished that the real parts of the poles are contained in the set of ratios $\left\{\frac{-v_{i}}{N_{i}}\right\}\cup\left\{-1,1\right\}$ where $\left\{(N_{i},v_{i})\right\}$ are the numerical data of an embedded resolution of singularities of the divisor $f^{-1}(0)\cup g^{-1}(0)$. However in the case of positive characteristic, this method cannot be applied. In this article we use a technique based on Newton polyhedra and a generalized $\pi$-adic stationary phase formula to give an explicit formula and the possible poles of  $Z\left(s,\chi,\frac{f}{g}\right)$ when $K$ has arbitrary characteristic and $f/g$ is non-degenerate with respect to certain Newton polyhedron. The local zeta functions attached to a class of non-degenerate rational functions (Laurent polynomials) were first studied in \cite{Leon} and \cite{Leon2}. The notion of non-degeneracy used here allows us to study the twisted local zeta functions attached to much larger class of rational functions.

	\begin{definition}
Let  $h(x_{1},\ldots,x_{n})=h(\boldsymbol{x})=\sum
		_{\boldsymbol{m}}c_{\boldsymbol{m}}\boldsymbol{x}^{\boldsymbol{m}}$ be a non constant polynomial with coefficients in $\mathcal{O}_{K}$ with
		 $h(\boldsymbol{0})=\boldsymbol{0}$ and $\mathrm{supp}(h)$ $:=$ $\left\{  \boldsymbol{m}%
		\in\mathbb{N}^{n};c_{\boldsymbol{m}}\neq0\right\}$. We define the \textit{Newton polyhedron }%
		$\Gamma=\Gamma\left( h\right)  $ of $h$ as the convex hull  of the set $$\bigcup_{\boldsymbol{m}\in supp(h)}\left(
		\boldsymbol{m}+\mathbb{R}_{+}^{n}\right)$$ in 		$\mathbb{R}_{+}^{n}:=\{x\in\mathbb{R};x\geqslant0\}$.
	\end{definition}
Every proper face of $\Gamma$ is the intersection of $\Gamma$ with a supporting hyperplane. For $\boldsymbol{k}\in\mathbb{R}_{+}^{n}$, we define  \textit{the first meet locus }$F(\boldsymbol{k},\Gamma)$ of
	$\boldsymbol{k}$ as \
	\[
	F(\boldsymbol{k},\Gamma):=\{\boldsymbol{x}\in\Gamma;\left\langle
	\boldsymbol{k},\boldsymbol{x}\right\rangle =d(\boldsymbol{k},\Gamma)\}.
	\] where 
	\[
	d(\boldsymbol{k},\Gamma)=\min_{x\in\Gamma}\left\langle \boldsymbol{k}%
	,\boldsymbol{x}\right\rangle.
	\]
The \textit{face function
	}$h_{\boldsymbol{k}}\left(  \boldsymbol{x}\right)  $\textit{\ of
	}$h(\boldsymbol{x})$\textit{\ with respect to }$\boldsymbol{k}$ is defined as 
	\[
	h_{\boldsymbol{k}}\left(  \boldsymbol{x}\right)  =\sum_{\boldsymbol{m}\in F(\boldsymbol{k}%
		,\Gamma)}c_{\boldsymbol{m}}\boldsymbol{x}^{\boldsymbol{m}}.
	\]
	For functions with subindices $h_{i}(\boldsymbol{x})$, we	will use the notation $h_{i,\boldsymbol{k}}(\boldsymbol{x})$ for the face
	function of $h_{i}(\boldsymbol{x})$\ with respect to $\boldsymbol{k}$.

Let $h_{i}$ be a non-constant polynomial with coefficients in $\mathcal{O}_{K}$ for $i=1,\ldots ,r$ , $r\leq n$.  We say that a polynomial mapping  $\boldsymbol{h}=(h_{1},\ldots,h_{r})$, with $\boldsymbol{h}\left(  \boldsymbol{0}\right)  =0$ is \textit{non-degenerate over} $\mathbb{F}_{q}$ \textit{ with respect to the Newton polyhedron} $\Gamma\left(
\boldsymbol{h}\right):=\Gamma\left(
{\textstyle\prod\nolimits_{i=1}^{r}}
h_{i}\left(  \boldsymbol{x}\right)  \right)$,  if for every vector $\boldsymbol{k}\in\mathbb{R}_{+}^{n}$ and for
		any non-empty subset $I\subseteq\left\{  1,\ldots,r\right\}  $, it verifies
		that
		\begin{equation}
		rank_{\mathbb{F}_{q}}\left[  \frac{\partial\overline{h}_{i,\boldsymbol{k}}%
		}{\partial x_{j}}\left(  \overline{\boldsymbol{z}}\right)  \right]  _{i\in
			I,\text{ }j\in\left\{  1,\ldots,n\right\}  }=Card(I) \label{Condition_2}%
		\end{equation}
		for any
		\begin{equation}
		\overline{\boldsymbol{z}}\in\left\{  \overline{\boldsymbol{z}}\in\left(
		\mathbb{F}_{q}^{\times}\right)  ^{n};\overline{h}_{i,\boldsymbol{k}}%
		(\overline{\boldsymbol{z}})=0\Leftrightarrow i\in I\right\},
		\label{Condition_1}%
		\end{equation}
		where $\overline{\cdot}$ denotes the reduction module $\pi\mathcal{O}_{K}$.
		
		 It is sufficient to verify the condition (\ref{Condition_2}) for $\boldsymbol{k=}b(\Delta)$ for all $\Delta
		\in\mathcal{F}(\boldsymbol{h})\cup\left\{  \boldsymbol{0}\right\}$, where $\mathcal{F}(\boldsymbol{h})$ is a simplicial polyhedral subdivision subordinate to $\Gamma(\boldsymbol{h})$ and $b(\Delta)$ is the \textit{barycenter} of $\Delta$. For further details on Newton polyhedra and polyhedral subdivisions the reader may consult  \cite[Section 3]{B-Z-0}, \cite{K-M-S}, \cite{K}, and \cite{Sturmfels}. 
\begin{definition}
	Let $f(x_{1},\ldots,x_{n}), g(x_{1},\ldots,x_{n})\in{\mathcal{O}}_{K}[x_{1},\ldots,x_{n}]\backslash\pi{\mathcal{O}}_{K}[x_{1},\ldots,x_{n}]$ two co-prime polynomials, $n\geq 2$. We say that the rational function $\frac{f}{g}$ is non-degenerate over $\mathbb{F}_{q}$ with respect to $\Gamma\left(\frac{f}{g}\right)$ if and only if the polynomial mapping $(f,g):K^{n}%
	\rightarrow K^{2}$ is non-degenerate over $\mathbb{F}_{q}$ with
	respect to $\Gamma\left(f,g\right)$.
\end{definition}	

In \cite{B-Z-0} an explicit formula for $Z\left(s,\frac{f}{g}\right)$ was stablished in the case for which $f/g$ is non-degenerate over $\mathbb{F}_{q}$ with respect to $\Gamma\left(\frac{f}{g}\right)$.
		In this work, we extend this result to the case of $Z\left(s,\chi, \frac{f}{g}\right)$ for any character $\chi$, see Theorem \ref{Theorem2}. Furthermore, we give the upper and lower bounds for the possible  negative and positive real parts of the poles in terms of $(t_{0},\ldots, t_{0})$, the intersection point of the diagonal with the boundary of the Newton polyhedron of $f/g$, see Propositions \ref{cotasup}, \ref{cotainf}. In the classical case, when $\chi=\chi_{triv}$, it is well known that the largest real negative pole different from $-1$ is $-1/t_{0}$ if $t_{0}>1$ and it implies that $Z(s,f)$ has always a real pole, see \cite{Hoo1}. In this case, it  remains true if  $(t_{0},\ldots, t_{0})$ is the intersection point with the boundary of the Newton polyhedron of $f$, see Theorem \ref{Diagonal1}. Similarly, in Theorem \ref{Diagonal2}, we prove that the smallest real positive pole different from $1$ is $1/t_{0}$ if $t_{0}>1$ and $(t_{0},\ldots, t_{0})$ is the intersection point of the diagonal with the boundary of the Newton polyhedron of $g$.  In Section \ref{Sect_7}, we show that, under certain conditions, the local zeta functions for non-degenerate rational functions attached  to the trivial character have a real pole, by describing the largest and the smallest negative and positive real poles, respectively. In addition, we get some conditions  on $\Gamma(f)$ and $\Gamma(g)$ that imply the existence of poles of $Z\left(s,\frac{f}{g}\right)$ and another ones for which the local zeta function does not have poles, see Remark \ref{Char}.

	\section{\label{Sect_3} Multivariate stationary phase formula}
	
	In this section we will stablish a generalization of the $\pi$-adic stationary phase formula. In the classical case it was introduced by Igusa in \cite{Igusa1} and it is very useful to compute local zeta functions for several types of polynomials, see e. g. \cite{Leon2}, \cite{W-Z}.  The  generalization introduced here coincides with the $\pi$-adic stationary phase formula in the case of one polynomial and one character.  In order to stablish this generalization we will first  compute some required integrals.

	\begin{lemma}
		\label{Lemma_igusa} Let $a\in K$, $c\in\mathbb{Z}$, $s\in\mathbb{C}$ with
		$\operatorname{Re}\left(  s\right)  >0$, $N\in\mathbb{Z}\backslash\left\{
		0\right\}  $ and $n\in{\mathbb{N}}\backslash\left\{  0\right\}  $, $\chi$ a
		character of ${\mathcal{O}}_{K}^{\times}$.Then
		\begin{align*}
		&  \int\limits_{a+\pi^{c}{\mathcal{O}}_{K}\backslash\left\{  0\right\}  }%
		\chi^{N}(ac\left(  x\right)  )\left\vert x\right\vert _{K}^{Ns+n-1}|dx|_{K}\\
		&  =\left\{
		\begin{array}
		[c]{lll}%
		\frac{(1-q^{-1})q^{-cNs-cn}}{(1-q^{-Ns-n}} &  & \text{$a\in\pi^{c}%
			{\mathcal{O}}_{K}$, $\chi^{N}=1$},\\
		&  & \\
		q^{-c}\chi^{N}(ac\left(  a\right)  )|a|_{K}^{Ns+n-1} &  &
		\begin{array}
		[c]{l}%
		a\notin\text{$\pi^{c}{\mathcal{O}}_{K}$ and}\\
		\text{ $\chi^{N}|_{U^{\prime}}=1$},%
		\end{array}
		\\
		&  & \\
		0 &  & \text{all other cases,}%
		\end{array}
		\right.
		\end{align*}
		in which $U^{\prime}=1+\pi^{c}a^{-1}{\mathcal{O}}_{K}$, and $|dx|_{K}$ denotes de Haar measure on $(K,+)$.
	\end{lemma}
	
	The proof of the lemma is a variation of the one given for Lemma 8.2.1 in
	\cite{Igusa}.\\
	
Let $n\geq 2$, $r\leq n$, and $\boldsymbol{h}=(h_{1},\ldots,h_{r}):K^{n}\rightarrow K^{r}$ be a polynomial mapping such that each $h_{i}%
	(\boldsymbol{x})\in\mathcal{O}_{K}[x_{1},\ldots,x_{n}]\backslash\pi
	\mathcal{O}_{K}[x_{1},\ldots,x_{n}]$.  The polynomial mapping $\overline{\boldsymbol{h}}:=(\overline{h}_{1},\ldots,\overline{h}_{r})$ denotes  the reduction module $\pi$ of the coefficients of all the polynomial components $h_{i}$.
	
	\begin{lemma}\label{Lemma1}
		Let $I=\{1,2,...,l\}$ with $l\leq r$,  $\boldsymbol{h}_{I}=(h_{1}(\boldsymbol{x}),h_{2}(\boldsymbol{x}),\ldots,h_{l}(\boldsymbol{x}))$, and $c\in\mathbb{N}\backslash\{0\}$ Suppose $\boldsymbol{a}\in\mathcal{O}_{K}^{n}$ such that $\overline{\boldsymbol{h}}_{I}(\overline{\boldsymbol{a}})=0$ and the Jacobian  matrix $Jac(\overline{\boldsymbol{h}}_{I},\overline{\boldsymbol{a}})$ has rank $l$, and that $\chi_{i}\neq \chi_{triv}$ for some $i\in I$. Then \[\int\limits_{\substack{{\boldsymbol{a}+(\pi^{c}\mathcal{O}}_{K})^{n}\backslash D_{K}}}\prod_{i=1}^{l}\chi_{i}(ac(h_{i}(\boldsymbol{x})))|h_{i}(\boldsymbol{x})|_{K}^{s_{i}}|d\boldsymbol{x}|_{K}=0\]
	\end{lemma}
	\begin{proof}
		By using that rank of $Jac(\overline{\boldsymbol{h}}_{I},\overline{\boldsymbol{a}%
		})$ is $l$, then the result follows from Lemma \ref{Lemma_igusa} by applying the change of variables 
		\begin{center}
			\begin{tabular}{ccccc}
				$\phi:$&$\mathcal{O}_{K}^{n}$&$\rightarrow $&$\mathcal{O}_{K}^{n}$\\
				&$(x_{1},\ldots,x_{n})$&$\mapsto$&$(y_{1},\ldots,y_{n})$&
			\end{tabular}
		\end{center}
		with
		\[
		\boldsymbol{y}_{i}=\phi_{i}(\boldsymbol{x}):=\left\{
		\begin{array}
		[c]{lll}%
		\frac{h_{i}\left(  \boldsymbol{a}+\pi^{c}\boldsymbol{x}\right)  -h_{i}\left(
			\boldsymbol{a}\right)  }{\pi^{c}} & \text{if} & i=1,\ldots,l,\\
		&  & \\
		x_{i} & \text{if} & l+1\leq i\leq n ,
		\end{array}
		\right.
		\]
		which gives a measure-preserving map from ${\mathcal{O}}_{K}^{n}$
		to itself, see  \cite[Lemma 7.4.3]{Igusa}.
	\end{proof}

	\subsection{$\pi$-adic stationary phase formula}
	Let $\overline{E}\subseteq \mathbb{F}_{q}^{n}$ and $E$ its preimage under the canonical homomorphism ${\mathcal{O}}^{n}_{K}\rightarrow \left({\mathcal{O}}_{K}/P_{K}\right)^{n}\cong \mathbb{F}_{q}^{n}$ and $\boldsymbol{h}=(h_{1},\ldots,h_{n})$ as above. For every subset $I\subseteq T:=\{1,2,\ldots, n\}$, we set
	 	\begin{equation}
	 	\overline{V}_{I}:=\left\{ \overline{\boldsymbol{z}}\in \mathbb{F}_{q}^{n};\ 
	 	\overline{h}_{i}(\overline{\boldsymbol{z}})=0\Leftrightarrow i\in I\right\} .
	 	\label{Notation}
	 	\end{equation}
	 	and for $I\neq \varnothing$ we define the singular locus of $\overline{\boldsymbol{h}}_{I}$ as the set 
	 	$$Sing_{\overline{\boldsymbol{h}}_{I}}({\mathbb{F}}%
	 	_{q})=\{\overline{z}\in	 	\overline{V}_{I}; Jac(\overline{\boldsymbol{h}}_{I},\overline{z}) \text{ has no maximal rank} \}.$$
	 	Let $S:=S(\boldsymbol{h}, E)$ be the subset of $\mathcal{O}_{K}^{n}$ mapped bijectively to the set $\bigcup_{\substack{ I\subseteq
	 			T  \\ I\neq \varnothing }}Sing_{\overline{\boldsymbol{h}}_{I}}({\mathbb{F}}%
	 	_{q})\cap \overline{E}.$ Furthermore, we set  $\boldsymbol{\chi}:=(\chi_{1},\ldots,\chi_{r})$, where $\chi_{i}$ is a multiplicative character of $\mathcal{O}_{K}^{\times
	 	}$. For our purpose we use two types of $\boldsymbol{\chi}$,  $\boldsymbol{\chi}_{triv}:=\left(  \chi_{triv},\ldots,\chi_{triv}\right)  $ and $\boldsymbol{\chi}=(\chi_{1},\ldots,\chi_{r})$ with  $\chi_{i}\neq\chi_{triv}$ and conductor $e_{\chi_{i}}=e$ for all $i\in T$, where $e$ is a fixed non-zero natural number. To refer to the second type we will write $\boldsymbol{\chi}\neq \boldsymbol{\chi}_{triv}$.  
Now,  we set 
	 	\begin{equation*}
	 	\nu (\overline{\boldsymbol{h}},E,\boldsymbol{\chi}):=\left\{
	 	\begin{array}
	 	[c]{lll}%
	 	q^{-n}Card\left\{ \overline{a}\in \overline{E}%
	 	|\overline{h_{i}}(\overline{a})\neq 0, i\in T\right\} . & \text{if} & \boldsymbol{\chi}=\boldsymbol{\chi}_{triv},\\
	 	&  & \\
	 	\sum\limits_{\substack{a\in E\\ \overline{\boldsymbol{a}}\mod\pi^{e}\\\overline{h}_{i}(\overline{\boldsymbol{a}})\neq 0, i\in T}} q^{-ne}\prod_{i=1}^{r}\chi_{i}(h_{i}(\boldsymbol{a})) &
	 	\text{if} & \boldsymbol{\chi}\neq \boldsymbol{\chi}_{triv},\\
	 	\end{array}
	 	\right.
	 	\end{equation*}
	 	 and for every non-empty subset $I\subseteq T$, we set 
	 	\begin{equation*}
	 	\sigma(\overline{\boldsymbol{h}}_{I},E,\boldsymbol{\chi}):=\left\{
	 	\begin{array}
	 	[c]{lll}%
	 	q^{-n}Card\left\{ \text{$\overline{a}%
	 		\in \overline{E}|$ $\overline{a}\in $}\overline{V}_{I},\text{$\overline{a}%
	 		\notin Sing_{\overline{\mathbf{h}}_{I}}({\mathbb{F}}_{q})$.}\right\} & \text{if} & \boldsymbol{\chi}=\boldsymbol{\chi}_{triv},\\
	 	&  & \\
0 &
	 	\text{if} & \boldsymbol{\chi}\neq \boldsymbol{\chi}_{triv}.\\
	 	\end{array}
	 	\right.
	 	\end{equation*}

	 Let  $\boldsymbol{s}=(s_{1},\ldots,s_{r})\in\mathbb{C}^{r},\ \boldsymbol{\chi}$,  and $\boldsymbol{h}$ as before.  We define
	 \[
	 Z_{E}(\boldsymbol{s},\boldsymbol{\chi},\boldsymbol{h})=\int\limits_{E}\prod_{i=1}^{r}\chi_{i}\left(
	 ac\left(  h_{i}(\boldsymbol{x})\right)  \right)  \left\vert h_{i}%
	 (\boldsymbol{x})\right\vert _{K}^{s_{i}}|d\boldsymbol{x}|_{K}.
	 \]
	 We write $Z(\boldsymbol{s},\boldsymbol{\chi},\boldsymbol{h})$ when $E=\mathcal{O}_{K}^{n}$. 
	 
	 The next result is a generalization of the $\pi$-adic stationary phase formula. 

\begin{theorem}
	\label{formula_estacionaria} With the preceding notation, 
	\begin{equation*}
	Z_{E}(\boldsymbol{s},\boldsymbol{\chi}, \boldsymbol{h})=\nu (\overline{\boldsymbol{h}},E, \boldsymbol{\chi})+\sum_{\substack{ I\subseteq T\\ I\neq\varnothing}}%
	\sigma (\overline{\boldsymbol{h}}_{I},E, \boldsymbol{\chi})\prod_{i\in I}\frac{(q-1)q^{-1-s_{i}}}{%
		1-q^{-1-s_{i}}}+Z_{S}(\boldsymbol{s},\boldsymbol{\chi}, \boldsymbol{h})
	\end{equation*}
\end{theorem}

\begin{proof}
	
	Since $E\backslash S$	and $S$ form a partition of $E$, 
	\begin{displaymath}
	Z_{E}(\boldsymbol{s},\boldsymbol{\chi}, \boldsymbol{h})=Z_{E\backslash S}(\boldsymbol{s}, \boldsymbol{\chi},\boldsymbol{h})+Z_{S}(\boldsymbol{s},\boldsymbol{\chi},\boldsymbol{h}).
	\end{displaymath}
	Now, notice that $Z_{E\backslash S}(\boldsymbol{s},\boldsymbol{\chi},\boldsymbol{h})$ can be expressed as a finite sum of
	integrals%
	\begin{equation}
\int\limits_{\substack{ 
			\boldsymbol{a}+(\pi^{e} {\mathcal{O}}_{K})^{n}}}%
	\prod_{i=1}^{r}\chi_{i}(ac(h_{i}(\boldsymbol{x})))|h_{i}(\boldsymbol{x})|_{K}^{s_{i}}|d\boldsymbol{x}|_{K},
	\label{Int_1}
	\end{equation}%
	where $\boldsymbol{a}$ runs through a fixed set of representatives $\mathcal{%
		R}$ of $E\backslash S \bmod \pi^{e}.$ Then $Z_{E\backslash S}(\mathbf{s}, \boldsymbol{\chi}, \mathbf{h})$ equals%
	\begin{gather}\label{Sum_1}
	\sum_{\substack{ \boldsymbol{a}\in \mathcal{R}  \\ \overline{\boldsymbol{a}}%
			\in \overline{V}_{\varnothing }}}\int\limits_{\boldsymbol{a}+(\pi^{e} {\mathcal{O%
		}}_{K})^{n}}\prod_{i=1}^{r}\chi_{i}(ac(h_{i}(\boldsymbol{x})))|h_{i}(\boldsymbol{x}%
	)|_{K}^{s_{i}}|d\boldsymbol{x}|_{K} \\
	+\sum_{\substack{ I\subseteq \left\{ 1,\ldots ,r\right\}  \\ I\neq
			\varnothing }}\sum_{\substack{ \boldsymbol{a}\in \mathcal{R}\text{, }%
			\overline{\boldsymbol{a}}\in \overline{V}_{I}  \\ \overline{\boldsymbol{a}}%
			\notin Sing_{\overline{\mathbf{h}}_{I}}({\mathbb{F}}_{q})}}\hspace{0.2cm}%
	\int\limits_{\boldsymbol{a}+(\pi^{e}{\mathcal{O}}_{K})^{n}}\prod_{i=1}^{r}\chi_{i}(ac(h_{i}(\boldsymbol{x})))|h_{i}(\boldsymbol{x})|_{K}^{s_{i}}|d\boldsymbol{x}|_{K}\nonumber
	\end{gather}
	With the convention that if $\overline{V}_{I}=\varnothing $ for $I\subseteq
	\left\{ 1,\ldots ,r\right\} $, then $\sum_{\substack{ \boldsymbol{a}\in 
			\mathcal{R} \\ \overline{\boldsymbol{a}}\in \overline{V}_{I}}}%
	\int\limits_{_{\substack{ \boldsymbol{a}+(\pi {\mathcal{O}}%
				_{K})^{n}}}}\cdot =0$. Thus we may assume that $\overline{V}%
	_{I}\neq \varnothing $.
	Now, if $\overline{\boldsymbol{a}}\in \overline{V}_{\varnothing}$, then 
	\begin{align*}
\int\limits_{\substack{ 
		\boldsymbol{a}+(\pi^{e} {\mathcal{O}}_{K})^{n}}}%
\prod_{i=1}^{r}\chi_{i}(ac(h_{i}(\boldsymbol{x})))|h_{i}(\boldsymbol{x})|_{K}^{s_{i}}|d\boldsymbol{x}|_{K}&= \int\limits_{\substack{ 
		\boldsymbol{a}+(\pi^{e} {\mathcal{O}}_{K})^{n}}}%
\prod_{i=1}^{r}\chi_{i}(ac(h_{i}(\boldsymbol{x})))|d\boldsymbol{x}|_{K}\\
&=q^{-ne}\int\limits_{\substack{ 
		 {\mathcal{O}}_{K}^{n}}}%
\prod_{i=1}^{r}\chi_{i}(ac(h_{i}(\boldsymbol{a}+\pi^{e}x)))|d\boldsymbol{x}|_{K}\\
&= 
q^{-ne}\prod_{i=1}^{r}\chi_{i}(h_{i}(\boldsymbol{a})),
\end{align*}
the last equality follows from the fact that $h_{i}(\boldsymbol{a})\in\mathcal{O}^{\times}_{K}$ and $e(\chi_{i})=e$ for all $i=1,\ldots,r$ and by applying Taylor's formula we obtain 
	$$\chi_{i}(ac(h_{i}(\boldsymbol{a}+\pi^{e}\boldsymbol{x})))=\chi_{i}(h_{i}(\boldsymbol{a})),$$ thus the first sum in (\ref{Sum_1}) equals $\nu(\overline{\boldsymbol{h}},E,\boldsymbol{\chi}).$
	
	Let $I$ be a non-empty subset of $\{1,2,\ldots, r\}$ and  $\overline{\boldsymbol{a}}\in \overline{V}_{I}$ such that $\overline{\boldsymbol{a}} \notin Sing_{\overline{\mathbf{h}}_{I}}(\mathbb{F}_{q})$. If $I=\{1,2,\ldots, r\}$, then by Lemma 1 in \cite{B-Z-0} and Lemma \ref{Lemma1} 
		\begin{equation}\label{In_1}
	\int\limits_{\substack{ 
			\boldsymbol{a}+(\pi^{e} {\mathcal{O}}_{K})^{n}}}%
	\prod_{i=1}^{r}\chi_{i}(ac(h_{i}(\boldsymbol{x})))|h_{i}(\boldsymbol{x})|_{K}^{s_{i}}|d\boldsymbol{x}|_{K}=\left\{
	\begin{array}
	[c]{lll}%
	0 & \text{if} & \boldsymbol{\chi}\neq\boldsymbol{\chi}_{triv},\\
	&  & \\
	q^{-ne}\prod\limits_{i=1}^{r}\frac{(q-1) q^{-1-s_{i}}}{{1-q^{-1-s_{i}}}} &
	\text{if} & \boldsymbol{\chi}=\boldsymbol{\chi}_{triv}.
	\end{array}
	\right.
	\end{equation}
	
	Now, we assume that $I\neq \{1,2,\ldots, r\}$, as above we can notice that  $$\chi_{i}(ac(h_{i}(\boldsymbol{a}+\pi^{e}\boldsymbol{x})))=\chi_{i}(h_{i}(\boldsymbol{a})) \text{ for $i\notin I$}.$$ Then $\int\limits_{\substack{ 
			\boldsymbol{a}+(\pi^{e} {\mathcal{O}}_{K})^{n}}}%
	\prod_{i=1}^{r}\chi_{i}(ac(h_{i}(\boldsymbol{x})))|h_{i}(\boldsymbol{x})|_{K}^{s_{i}}|d\boldsymbol{x}|_{K}$ equals 
		\begin{equation}\label{In_2}
 q^{-ne}\prod_{i\notin I}\chi_{i}(h_{i}(\boldsymbol{a})) \int\limits_{\substack{ 
			{\mathcal{O}}_{K}^{n}}}
	\prod_{i\in I}\chi_{i}(ac(h_{i}(\boldsymbol{a}+\pi^{e}\boldsymbol{x})))|h_{i}(\boldsymbol{a}+\pi^{e}\boldsymbol{x})|_{K}^{s_{i}}|d\boldsymbol{x}|_{K}.
	\end{equation}
	Then by using that $\overline{\boldsymbol{a}}\in \overline{V}_{I}$ with  $\overline{\boldsymbol{a}} \notin Sing_{\overline{\boldsymbol{h}}_{I}}(\mathbb{F}_{q})$ and Lemma 1 in \cite{B-Z-0} and Lemma \ref{Lemma1} it follows that 
\begin{equation}\label{In_3}
\int\limits_{\substack{ 
		{\mathcal{O}}_{K}^{n}}}%
\prod_{i\in I}\chi_{i}(ac(h_{i}(\boldsymbol{a}+\pi^{e}\boldsymbol{x})))|h_{i}(\boldsymbol{a}+\pi^{e}\boldsymbol{x})|_{K}^{s_{i}}|d\boldsymbol{x}|_{K}=\left\{
\begin{array}
[c]{lll}%
0 & \text{if} & \boldsymbol{\chi}\neq\boldsymbol{\chi}_{triv},\\
&  & \\
\prod\limits_{i\in I}\frac{(q-1) q^{-1-s_{i}}}{{1-q^{-1-s_{i}}}} &
\text{if} & \boldsymbol{\chi}=\boldsymbol{\chi}_{triv}.
\end{array}
\right.
\end{equation}
	Thus by equations (\ref{In_1}), (\ref{In_2}), and (\ref{In_3}), it follows that the second sum of (\ref{Sum_1}) equals 
	\begin{equation*}
\sum_{\substack{ I\subseteq \left\{ 1,\ldots ,r\right\}  \\ I\neq
		\varnothing }}\sigma(\overline{\boldsymbol{h}}_{I},E,\boldsymbol{\chi})\prod_{i\in I}\frac{(q-1)q^{-1-s_{i}}}{1-q^{-1-s_{i}}}. 
	\end{equation*}
\end{proof}
\begin{remark}
The multivariate stationary phase formula given in Theorem \ref{formula_estacionaria} can be generalized in a similar way to the case of characters with different conductors or even combinations of trivial and non-trivial  characters  by taking  $e=\max\{e_{\chi_{i}}\}$. We do not include it here because that generalization is not necessary in this work.
\end{remark}

	\section{\label{Sect_5} Explicit formula  for local zeta functions attached to rational functions}
In order to give an explicit formula for twisted local zeta functions attached to rational functions, we extend the result given in  \cite[Theorem 1]{B-Z-0} to the case of multivariate twisted local zeta functions  $Z(\boldsymbol{s},\boldsymbol{\chi}, \boldsymbol{h})$ by using the multivariate stationary phase formula stablished in the last section. In addition, the following theorem  extends the result given in \cite[Theorem 3.4]{Hoo1} by Hoornaert.

	\begin{theorem}
		\label{Theorem1} Let $e\in\mathbb{N}\backslash\{0\}$ and $r\leq n$. Assume that $\boldsymbol{h}=(h_{1},\ldots,h_{r})$ is 
		non-degenerated polynomial mapping over ${\mathbb{F}}_{q}$ with respect to
		$\Gamma(\boldsymbol{h})$, and $\boldsymbol{\chi}=(\chi_{1},\ldots,\chi_{r})$ with $\boldsymbol{\chi}=\boldsymbol{\chi_{_{triv}}}$ or $\chi_{i}\neq \chi_{triv}$ and $e(\chi_{i})=e$ for $i=1,\ldots, r$. Fix a simplicial
		polyhedral subdivision $\mathcal{F}(\boldsymbol{h})$ subordinate to
		$\Gamma(\boldsymbol{h})$ and let $\boldsymbol{w}_{1},\ldots,\allowbreak\boldsymbol{w}_{l}\in
		{\mathbb{N}}^{n}\backslash\left\{  \mathbf{0}\right\}$ be the generators of a cone $\Delta\in \mathcal{F}(\boldsymbol{h})$. Then $Z(\mathbf{s},\boldsymbol{\chi},\boldsymbol{h}%
		)$ has a meromorphic continuation to $\mathbb{C}^{r}$ as a rational function
		in the variables $q^{-s_{i}}$, $i=1,\ldots,r$. In addition, the following explicit formula holds:
		\[
		Z(\boldsymbol{s},\boldsymbol{\chi},\boldsymbol{h})=\sum
		_{\substack{\text{$\Delta\in\mathcal{F}(\boldsymbol{h})\cup \{\boldsymbol{0}\}$}}}L_{\Delta
		}(\boldsymbol{s},\boldsymbol{\chi},\boldsymbol{h})S_{\Delta},
		\]
		where 
				\[
		L_{\Delta}(\boldsymbol{s},\boldsymbol{\chi},\boldsymbol{h})=\nu (\overline{\mathbf{h}}_{\Delta},(\mathcal{O}_{K}^{\times})^{n}, \boldsymbol{\chi})+\sum_{\substack{ I\subseteq T\\ I\neq\varnothing}}%
		\sigma (\overline{\mathbf{h}}_{I,\Delta},(\mathcal{O}_{K}^{\times})^{n}, \boldsymbol{\chi})\prod_{i\in I}\frac{(q-1)q^{-1-s_{i}}}{%
			1-q^{-1-s_{i}}},
		\]	
 
	 $S_{\{\boldsymbol{0}\}}=1$, and 
		\[
		S_{\Delta}=\frac{\sum_{\boldsymbol{t}}q^{-\sigma(\boldsymbol{t})-\sum
				_{i=1}^{r}d(\boldsymbol{t},\Gamma(h_{i}))s_{i}}}{(1-q^{-\sigma(\boldsymbol{w}%
				_{1})-\sum_{i=1}^{r}d(\boldsymbol{w}_{1},\Gamma(h_{i}))s_{i}})\cdots
			(1-q^{-\sigma(\boldsymbol{w}_{l})-\sum_{i=1}^{r}d(\boldsymbol{w}_{l}%
				,\Gamma(h_{i}))s_{i}})},
		\]
		with $\boldsymbol{t}$ runs through the elements of the set
		\begin{equation}\label{Subset_1}
		{\mathbb{Z}}^{n}\cap\left\{  \sum_{i=1}^{l}\lambda_{i}\boldsymbol{w}%
		_{i};\text{ $0<\lambda_{i}\leq1$ for $i=1,\ldots,l$}\right\} .
		\end{equation}
	\end{theorem}
	
	\begin{proof}

Notice that 
\[
{\mathbb{R}}_{+}^{n}=\left\{  \boldsymbol{0}\right\}
{\textstyle\bigsqcup}
{\textstyle\bigsqcup\nolimits_{\Delta\in{\mathcal{F}}(\boldsymbol{h})}}
\Delta.
\]
Hence%
\begin{align*}
Z(\boldsymbol{s},\boldsymbol{\chi},\boldsymbol{h})&=\int\limits_{({\mathcal{O}}_{K}^{\times})^{n}}\prod_{i=1}^{r}\chi_{i}(ac(h_{i}(\boldsymbol{x})))|h_{i}(\boldsymbol{x})|_{K}^{s_{i}}|d\boldsymbol{x}|_{K}\\
&+\sum
_{\substack{\text{$\Delta\in\mathcal{F}(\boldsymbol{h})$}}}\hspace{0.1cm}
\sum_{\boldsymbol{k}\in{\mathbb{N}}^{n}\cap\Delta}\hspace{0.3cm}
\int\limits_{\substack{\left\{  \boldsymbol{x}\in{\mathcal{O}}_{K}
		^{n};ord(\boldsymbol{x})=\boldsymbol{k}\right\}}}\prod_{i=1}^{r}\chi_{i}(ac(h_{i}(\boldsymbol{x})))|h_{i}(x)|_{K}^{s_{i}}|d\boldsymbol{x}|_{K}.
\end{align*}
By making the following change of variables
  $x_{j}=\pi^{k_{j}}u_{j}$ with $u_{j}\in{\mathcal{O}}_{K}^{\times}$, we obtain 
$|d\boldsymbol{x}|_{K}=q^{-\sigma(\boldsymbol{k})}|d\boldsymbol{u}|_{K}$, and $\boldsymbol{x}^{\boldsymbol{m}}=\pi^{\left\langle\boldsymbol{k},\boldsymbol{m}\right\rangle }\boldsymbol{u}^{m}$.
\begin{displaymath}
h_{i}(\boldsymbol{x})   =\pi^{d(\boldsymbol{k},\Gamma(h_{i}))}
(h_{i,\boldsymbol{k}}(\boldsymbol{u})+\pi^{l}\widetilde{h}_{i,\boldsymbol{k}
}(\boldsymbol{u}))
\end{displaymath}
where $\widetilde{h}_{i,\boldsymbol{k}}(\boldsymbol{u})\in \mathcal{O}^{\times}_{K}[\boldsymbol{x}]$ and $l>0$.  Note that
$h_{i,\boldsymbol{k}}(\boldsymbol{u})$ does not depend on $\boldsymbol{k}%
\in\Delta$, thus we will use $h_{i,\boldsymbol{k}}(\boldsymbol{u}%
)$ instead of $h_{i,b(\Delta)}(\boldsymbol{u})$, where $b(\Delta)$ is the barycenter of $\Delta$, hence 
\begin{displaymath}
h_{i}(\boldsymbol{x})  
=\pi^{d(\boldsymbol{k},\Gamma(h_{i}))}(h_{i,b(\Delta)}(\boldsymbol{u}%
)+\pi^{l}\widetilde{h}_{i,\boldsymbol{k}}(\boldsymbol{u})).
\end{displaymath} 
Then the integral $\int\limits_{\substack{\left\{  \boldsymbol{x}\in{\mathcal{O}}_{K}
		^{n};ord(\boldsymbol{x})=\boldsymbol{k}\right\}}}\prod_{i=1}^{r}\chi_{i}(ac(h_{i}(\boldsymbol{x})))|h_{i}(x)|_{K}^{s_{i}}|d\boldsymbol{x}|_{K}$ becomes to 
\begin{align*}
\sum_{\boldsymbol{k}\in{\mathbb{N}}^{n}\cap\Delta}&
\hspace{0.3cm}q^{-\sigma(\boldsymbol{k})-\sum_{i=1}^{r}d(\boldsymbol{k}
	,\Gamma(h_{i}))s_{i}} \\
&\times \int\limits_{({\mathcal{O}}
	_{K}^{\times})^{n}}\prod_{i=1}^{r}\chi_{i}(ac(h_{i,b(\Delta)}
(u)+\pi^{l}\widetilde{h}_{i,\boldsymbol{k}}(\boldsymbol{u})))|h_{i,b(\Delta)}
(u)+\pi^{l}\widetilde{h}_{i,\boldsymbol{k})}(\boldsymbol{u})|_{K}^{s_{i}
}|d\boldsymbol{u}|_{K}
\end{align*}
Now, since $\boldsymbol{h}$ is non degerate over $\mathbb{F}_{q}$ with respect to the Newton polyhedron $\Gamma(\boldsymbol{h})$ then for any $\Delta\in\mathcal{F}(\boldsymbol{h})\cup\{\boldsymbol{0}\}$ the polynomial mapping  $\boldsymbol{\boldsymbol{h}}_{\Delta}=(h_{1,b(\Delta)},\ldots, h_{r,b(\Delta)})$ satisfies $Sing_{\overline{h}_{\Delta, I}}(\mathbb{F}_{q})\cap(\mathbb{F}_{q}^{\times})^{n}=\varnothing$ for any non-empty subset  $I\subseteq \{1,2,\ldots, r\}$.
Thus by Theorem \ref{formula_estacionaria}, 
\[
Z(\boldsymbol{s},\boldsymbol{\chi},\boldsymbol{h})=\sum
_{\substack{\text{$\Delta\in\mathcal{F}(\boldsymbol{h})\cup \{\boldsymbol{0}\}$}}}L_{\Delta
}(\boldsymbol{s},\boldsymbol{\chi},\boldsymbol{h})\sum_{\boldsymbol{k}\in{\mathbb{N}}^{n}\cap\Delta}
\hspace{0.3cm}q^{-\sigma(\boldsymbol{k})-\sum_{i=1}^{r}d(\boldsymbol{k}
,\Gamma(h_{i}))s_{i}}.
\]
Finally, $\sum_{\boldsymbol{k}\in{\mathbb{N}}^{n}\cap\Delta}
\hspace{0.3cm}q^{-\sigma(\boldsymbol{k})-\sum_{i=1}^{r}d(\boldsymbol{k}
	,\Gamma(h_{i}))s_{i}}=S_{\Delta}$  follows from the proof of Theorem  1 in \cite{B-Z-0}.
\end{proof}

\subsection{Explicit formula for $Z\left(s,\chi,\frac{f}{g}\right)$}

Let	$\frac{f}{g}$ be a non-degenerate rational function over $F_{q}$ with
		respect to $\Gamma\left(\frac{f}{g}\right)$ and let $\mathcal{D}\left(\frac{f}{g}\right)$ the set
	of primitive vectors in ${\mathbb{N}}^{n}\backslash\left\{  \mathbf{0}%
	\right\}  $ perpendicular to the facets of $\Gamma\left(\frac{f}{g}\right)$. Let 	$T_{+}$, $T_{-}$, $\alpha$, $\beta$, 	$\widetilde{\alpha}$ and $\widetilde{\beta}$ as in Section  5 in \cite{B-Z-0}.
	
	Let	$\chi$ be a multiplicative character of ${\mathcal{O}}_{K}^{\times}$, we define the local zeta function attached to $\left(\frac{f}{g},\chi\right)$ as
	\[
	Z\left(s,\chi,\frac{f}{g}\right)=Z(s,-s,\chi,\chi^{-1},f,g),\ s\in\mathbb{C},
	\]
	where $Z\left(s_{1},s_{2},\chi_{1},\chi_{2},f,g\right)$ denotes the meromorphic
	continuation of the local zeta function attached to the polynomial mapping
	$(f,g)$, see Theorem \ref{Theorem1}.
	
	\begin{theorem}
		\label{Theorem2} Let $\frac{f}{g}$ and $\chi$ as above, and let 		$\mathcal{F}\left(\frac{f}{g}\right)$ be a fixed simplicial polyhedral subdivision of $\mathbb{R}_{+}^{n}$ subordinate to
		$\Gamma\left(\frac{f}{g}\right)$ and let $\boldsymbol{w}_{1},\ldots,\allowbreak\boldsymbol{w}_{l}\in
		{\mathbb{N}}^{n}\backslash\left\{  \mathbf{0}\right\}$ be the generators of a cone $\Delta\in \mathcal{F}(\boldsymbol{h})$. Then the following assertions hold:
		
		\noindent(i) $Z\left(s,\chi,\frac{f}{g}\right)$ has a meromorphic continuation to the
		whole complex plane as a rational function of $q^{-s}$ and the following
		explicit formula holds:
		\[
		Z\left(s,\chi,\frac{f}{g}\right)=\sum_{\text{$\Delta\in\mathcal{F}\left(\frac{f}{g}\right)\cup$%
			}\left\{  \mathbf{0}\right\}  }L_{\Delta}\left(s,\chi,\frac{f}{g}\right)S_{\Delta}(s),
		\]
		where 
		\begin{align*}
		L_{\Delta}\left(s,\chi,\frac{f}{g}\right)&= \nu \left(\frac{f_{\Delta}}{g_{\Delta}}, \chi\right)-N_{\Delta,\left\{
	f\right\}  }\frac{1-q^{-s}}{1-q^{-1-s}}-N_{\Delta,\left\{  g\right\}  }%
\frac{1-q^{s}}{1-q^{-1+s}} \\
&-N_{\Delta,\left\{  f,g\right\}  }\frac{(1-q^{-s})(1-q^{s}%
	)}{q(1-q^{-1-s})(1-q^{-1+s})}
		\end{align*}
		with
			 	\begin{equation*}
		\nu \left(\frac{f_{\Delta}}{g_{\Delta}}, \chi\right)=\left\{
		\begin{array}
		[c]{lll}%
		q^{-n}(q-1)^{n}  & \text{if} & \chi=\chi_{triv},\\
		&  & \\
		\sum\limits_{\substack{\boldsymbol{a}\in(\mathcal{O}_{K}^{\times})^{n} \\ \overline{\boldsymbol{a}}\mod\pi^{e_{\chi}}\\\overline{f}_{b(\Delta)}(\overline{\boldsymbol{a}})\neq 0, \hspace{.2cm}\overline{g}_{b(\Delta)}(\overline{\boldsymbol{a}})\neq 0}} q^{-ne_{\chi}}\chi\left(\frac{f_{b(\Delta)}\left(\boldsymbol{a}\right)}{g_{b(\Delta)}\left(  \boldsymbol{a}\right)}\right)  &
		\text{if} & \chi\neq \chi_{triv},\\
		\end{array}
		\right.
		\end{equation*}
		where $e_{\chi}$ is the conductor of $\chi$, and 
		$$N_{\Delta,\{f\}}   =\sigma (\overline{f}_{\Delta},(\mathcal{O}_{K}^{\times})^{n}, \boldsymbol{\chi}),N_{\Delta, \{g\}}   =\sigma (\overline{g}_{\Delta},(\mathcal{O}_{K}^{\times})^{n}, \boldsymbol{\chi}),$$  and 
$$N_{\Delta, \{f,g\}}=\sigma ((\overline{f}_{\Delta}, \overline{g}_{\Delta}),(\mathcal{O}_{K}^{\times})^{n}, \boldsymbol{\chi})$$ for 
		$\boldsymbol{\chi}=(\chi,\chi)$, and 
		\[
		S_{\Delta}(s)=\frac{\sum_{\boldsymbol{t}}q^{-\sigma(\boldsymbol{t}%
				)-(d(\boldsymbol{t},\Gamma(f))-d(\boldsymbol{t},\Gamma(g))){s}}}{\prod
			_{i=1}^{l}(1-q^{-\sigma(\boldsymbol{w}_{i})-(d(\boldsymbol{w}_{i}%
				,\Gamma(f))-d(\boldsymbol{w}_{i},\Gamma(g))){s}})},
		\]
		where $\boldsymbol{t}$ runs through the elements of the set (\ref{Subset_1}).
		
		\noindent(ii) $Z\left(s,\chi,\frac{f}{g}\right)$ is a holomorphic function on
		$\widetilde{\beta}<\operatorname{Re}(s)<\widetilde{\alpha}$ if $\chi=\chi_{triv}$, or on  $\beta<\operatorname{Re}(s)<\alpha$ if $\chi\neq \chi_{triv}$ and on these bands 
		it verifies that
		\begin{equation}
		Z\left(s,\chi,\frac{f}{g}\right)=\int\limits_{{\mathcal{O}}_{K}^{n}\backslash D_{K}}%
		\chi\left(  ac\left(  \frac{f(x)}{g(x)}\right)  \right)  \left\vert
		\frac{f(x)}{g(x)}\right\vert ^{s}|dx|; \label{integral}%
		\end{equation}
		
	\end{theorem}
	
	\begin{proof}
		(i) It follows from Theorem \ref{Theorem1} as follows: we take $r=2$, $\chi_{1}=\chi$, $\chi_{2}=\chi
		^{-1}$, $h_{1}=f_{b\left(  \Delta\right)  }$ and $h_{2}=g_{b\left(
			\Delta\right)  }$ for $\Delta\in\mathcal{F}(\frac{f}{g})\cup$$\left\{
		\mathbf{0}\right\}  $, with the convention that if $b\left(  \Delta\right)
		=b\left(  \mathbf{0}\right)  =\mathbf{0}$, then $h_{1}=f$ and $h_{2}=g$
		
		(ii) From the explicit formula given in (i), follows that the rational function
		$Z(s,\chi,\frac{f}{g})=Z(s,-s,\chi,\chi^{-1},f,g)$ is holomorphic on the band
		$\widetilde{\beta}<\operatorname{Re}(s)<\widetilde{\alpha}$ if  $\chi=\chi_{triv}$, or on  $\beta<\operatorname{Re}(s)<\alpha$ if $\chi\neq \chi_{triv}$, and then
		$Z(s,\chi,\frac{f}{g})$ is given by integral (\ref{integral}) because
		$Z(s_{1},s_{2},\chi_{1},\chi_{2},f,g)$ agrees with an integral on its domain
		of holomorphy.
	\end{proof}
\begin{remark}
	We notice that in the expression  $\nu\left(\frac{f_{\Delta}}{g_{\Delta}}, \chi\right)$ when $\chi\neq \chi_{_{triv}},$ the conditions $\overline{f}_{b(\Delta)}(\overline{\boldsymbol{a}})\neq 0$ and  $\overline{g}_{b(\Delta)}(\overline{\boldsymbol{a}})\neq 0$ implies that the quotient $\frac{f_{b(\Delta)}\left(  \boldsymbol{a}\right)}{g_{b(\Delta)}\left(  \boldsymbol{a}\right)}$ is a unit in $\mathcal{O}_{K}$.
\end{remark}
	
\section{The Candidate Poles of $Z\left(s,\chi,\frac{f}{g}\right)$ and the poles determined by the diagonal}\label{Sect_6}
	In this section we use all the notation introduced in Section \ref{Sect_5}. We fix two co-prime polynomials $f(\boldsymbol{x})$, $g(\boldsymbol{x}%
	)\in{\mathcal{O}}_{K}[x_{1},\ldots,x_{n}]\backslash\pi{\mathcal{O}}_{K}%
	[x_{1},\ldots,x_{n}]$ with $n\geq 2$ and $f(\boldsymbol{0})=g(\boldsymbol{0})=0$, and also we fix a  simplicial polyhedral subdivision $\mathcal{F}\left(  \frac
	{f}{g}\right)  $ of $\mathbb{R}_{+}^{n}$ subordinate to $\Gamma\left(
	\frac{f}{g}\right)$.
\begin{proposition}\label{PropPoles}
	Suppose that $f,g$ satisfy the conditions of Theorem \ref{Theorem2}. Let $\gamma_{1}$,\ldots,  $\gamma_{l}$ be all the facets of $\Gamma\left(\frac{f}{g}\right)$ and let $\boldsymbol{w}_{1},\ldots \boldsymbol{w}_{l}\in\mathbb{N}\backslash\{0\}$ be the unique primitive vectors that are perpendicular to $\gamma_{1}$,\ldots,  $\gamma_{l}$ respectively. Then
	
		\noindent(i) If $\chi=\chi_{triv}$ and $s$ is a pole of the meromorphic
continuation of $Z\left(s,\frac{f}{g}\right)$, then 
\begin{align*}
&s=1+\frac{2\pi\sqrt{-1}k}{\ln{q}} \text{ with $k\in{\mathbb{Z}}$ or}\\
&s=  -1+\frac{2\pi
	\sqrt{-1}k}{\ln{q}}  \text{ with $k\in{\mathbb{Z}}$ or}\\
&s=\frac{\sigma(\boldsymbol{w_{i}}%
	)}{d(\boldsymbol{w_{i}},\Gamma(g))-d(\boldsymbol{w}_{i},\Gamma(f))}+\frac{2\pi
	\sqrt{-1}k}{\left\{  d(\boldsymbol{w_{i}},\Gamma(g))-d(\boldsymbol{w}_{i}%
	,\Gamma(f))\right\}  \ln{q}} .
\end{align*}
with $k\in{\mathbb{Z}}$, $i\in\{1,2,\ldots,l\}$ and $d(\boldsymbol{w_{i}},\Gamma(g))-d(\boldsymbol{w}_{i}%
,\Gamma(f))\neq 0$.

	\noindent(ii) If $\chi\neq\chi_{triv}$ and $s$ is a pole of the meromorphic
continuation of $Z(s,\chi,\frac{f}{g})$, then
\[s=\frac{\sigma(\boldsymbol{w}_{i}%
	)}{d(\boldsymbol{w}_{i},\Gamma(g))-d(\boldsymbol{w}_{i},\Gamma(f))}+\frac{2\pi
	\sqrt{-1}k}{\left\{  d(\boldsymbol{w}_{i},\Gamma(g))-d(\boldsymbol{w}_{i}%
	,\Gamma(f))\right\}  \ln{q}} .\]
with $k\in{\mathbb{Z}}$, $i\in\{1,2,\ldots,l\}$ and $d(\boldsymbol{w}_{i},\Gamma(g))-d(\boldsymbol{w}_{i}%
,\Gamma(f))\neq 0$.
\end{proposition}
\begin{proof}
	This result is a direct consequence of the explicit formula in Theorem \ref{Theorem2}.
	\end{proof}
We will call to each $s$ described above a \textbf{candidate pole} of $Z\left(s,\chi,\frac{f}{g}\right)$. 

Notice that if $\Gamma(f)=\Gamma(g)$, then $Z\left(s,\chi,\frac{f}{g}\right)$ does not have poles coming from $S_{\Delta}$. Then along this section we assume that $\Gamma(f)\neq\Gamma(g)$.

\subsection{The expected order of a candidate pole}
Let $f,g$, and $\chi$ as in Theorem \ref{Theorem2}. For any $k\in \mathbb{Q}\backslash\{0\}$ we put
	\[
\mathcal{P}(k):=\left\{
\begin{array}
[c]{lll}%
\left\{  \boldsymbol{w}\in T_{-};\frac{\sigma
	(\boldsymbol{w})}{d(\boldsymbol{w},\Gamma(g))-d(\boldsymbol{w},\Gamma
	(f))}=k\right\}  &
\text{if} & k<0,\\
&  & \\
\left\{  \boldsymbol{w}\in T_{+};\frac{\sigma
	(\boldsymbol{w})}{d(\boldsymbol{w},\Gamma(g))-d(\boldsymbol{w},\Gamma
	(f))}=k\right\}  & \text{if} & k>0 ,%
\end{array}
\right.
\]
and for $m\in\mathbb{N}$ with $1\leq m\leq n$,
\[
\mathcal{M}_{m}(k):=\left\{  \text{$\Delta\in\mathcal{F}\left(  \frac
	{f}{g}\right)  $; $\Delta$ has exactly $m$ generators belonging to
}\mathcal{P}(k)\right\}  ,
\]
 $\rho(k):=\max\left\{  m;\mathcal{M}_{m}(k)\neq\varnothing\right\} $.
 
 If $s$ is a candidate pole of $Z(s,\chi,\frac{f}{g})$,  we set $\mathcal{P}(s):=\mathcal{P}(Re(s))$, $\mathcal{M}_{m}(s):=\mathcal{M}_{m}(Re(s))$ and $\rho(s):=\rho(Re(s))$. 
\begin{definition}
 \textbf{The expected order of a candidate pole} $s$ is $\rho(s)$  if $Re(s)\neq -1,1$ and  $\chi=\chi_{triv}$ or $\chi\neq \chi_{triv}$ ; otherwise if $Re(s)=-1$ or $1$ and $\chi=\chi_{triv}$, it will be $1$, $\rho(s)$ or  $\rho(s)+1$.
\end{definition} This definition follows from the explicit form of $Z(s,\chi,\frac{f}{g})$ given in Theorem \ref{Theorem2}. Notice that the actual order of a pole is less than or equal than the expected order.

 We recall that in the case $T_{-}\neq\varnothing$, $\beta$ is the largest possible negative real part of the poles of
$Z(s,\chi,\frac{f}{g})$ coming from $S_{\Delta}$ and we set $\rho:=\rho(\beta)$.
Similarly, in the case $T_{+}\neq\varnothing$, $\alpha$ is the smallest possible positive real part of the poles of $Z(s,\chi,\frac{f}{g})$ coming from $S_{\Delta}$ and we set $\kappa:=\rho(\alpha)$.

\subsection{The pole determined by the diagonal $D=\{(t,\ldots,t):t\in\mathbb{R}\}$}
It is well known that in the case of local zeta functions attached to one polynomial, the `intersection point' of the diagonal $D=\{(t,\ldots,t):t\in\mathbb{R}\}$ with the boundary of the Newton polyhedron is the largest real candidate pole different from $-1$, see e.g.  \cite{Hoo1},  \cite{Varchenko}. But in this case of local zeta functions for rational functions, due to the forms of the poles, this occur only for special cases, see Theorems \ref{Diagonal1} and \ref{Diagonal2}. 
\begin{definition}
	Let $f$, $g$ be non-zero polynomials as above.  Let  $(t_{0},\ldots,t_{0})$ be the unique intersection point of the diagonal $D=\{(t,\ldots,t):t\in\mathbb{R}\}$ with the boundary of the Newton polyhedron  $\Gamma\left(\frac{f}{g}\right)$, $\tau_{0}$ be the smallest face of   $\Gamma\left(\frac{f}{g}\right)$ such that $(t_{0},\ldots,t_{0})\in \tau_{0}$. Let $\gamma_{1},\ldots,\gamma_{e}$ be all the facets such that $\tau_{0}\subseteq \gamma_{i}$ for all $i$, and $\boldsymbol{w}_{1},\ldots, \boldsymbol{w}_{e}\in\mathbb{N}\backslash\{0\}$ be the unique primitive vectors that are perpendicular to $\gamma_{1},\ldots,\gamma_{e}$, respectively. We denote by $\mathcal{D}(t_{0})$ the set of all of these primitive vectors and we set 
	$ \mathcal{D}_{-}(t_{0}):=\mathcal{D}(t_{0})\cap T_{-}$ and $ \mathcal{D}_{+}(t_{0}):=\mathcal{D}(t_{0})\cap T_{+}$.
\end{definition}
The following Proposition is similar to Lemma 5.3 in \cite{DenefHoor}.
\begin{proposition}\label{Prop3}
 For any $\boldsymbol{a}\in\mathbb{R}_{+}^{n}$ we have $\sigma(\boldsymbol{a})-(d(\boldsymbol{a},\Gamma(g))+d(\boldsymbol{a},\Gamma(f)))(1/t_{0})\geq 0$ with equality if and only if $\tau_{0}\subseteq  F\left(\boldsymbol{a},\Gamma\left(\frac{f}{g}\right)\right)$.
\end{proposition}
\begin{proof}
Because of $(t_{0},\ldots,t_{0})\in \Gamma((f,g))$ and by  definition of $d(\boldsymbol{a},\Gamma((f,g)))$, it follows that $(t_{0},\ldots,t_{0})\cdot \boldsymbol{a}\geq d(\boldsymbol{a},\Gamma((f,g)))$. Hence, $\sigma(\boldsymbol{a})-d(\boldsymbol{a},\Gamma((f,g)(1/t_{0})\geq 0$.  
The first result follows from the fact  $d(\boldsymbol{a},\Gamma((f,g)))=d(\boldsymbol{a},\Gamma(g))+d(\boldsymbol{a},\Gamma(f))$, see Remark 3 in \cite{B-Z-0}. Furthermore, $\sigma(\boldsymbol{a})-d(\boldsymbol{a},\Gamma(f,g))(1/t_{0})= 0$ if and only if $(t_{0},\ldots,t_{0})\in F(\boldsymbol{a},\Gamma\left( (f, g)\right)  ) $ if and only if $\tau_{0}\subseteq  F(\boldsymbol{a},\Gamma\left((f,g)\right))$, because $\tau_{0}$ is the smallest face that contains  $(t_{0}\ldots,t_{0})$.
	\end{proof}
\begin{corollary}\label{Cor1}
	Let $\boldsymbol{w}\in \mathcal{D}(t_{0})$. Then, $d(\boldsymbol{w},\Gamma(g))=0$ if and only if  $(t_{0},\ldots,t_{0})\in F(\boldsymbol{w},\Gamma(f))$. Similarly, $d(\boldsymbol{w},\Gamma(f))=0$ if and only if $(t_{0},\ldots,t_{0})\in F(\boldsymbol{w},\Gamma(g))$.
\end{corollary}
\begin{proof}
By using Proposition \ref{Prop3} and the definition of $F(\boldsymbol{w},\Gamma)$, the result follows from the following equivalent propositions, $d(\boldsymbol{w},\Gamma(g))=0$ if and only if $\sigma(\boldsymbol{w})-d(\boldsymbol{w},\Gamma(f))(1/t_{0})= 0$ if and only if $(t_{0},\ldots,t_{0})\in F(\boldsymbol{w},\Gamma\left( f\right)) $. A similar argument shows  $d(\boldsymbol{w},\Gamma(f))=0$ if and only if $(t_{0},\ldots,t_{0})\in F(\boldsymbol{w},\Gamma(g))$.
	\end{proof}

The following propositions stablish that all the possible negative (resp., positive) real parts of the poles coming from $S_{\Delta}$ have $-1/t_{0}$ as an upper bound (resp., $1/t_{0}$ as a lower bound). Furthermore, the result stablish in Proposition \ref{cotasup} is similar to the one given in the case of local zeta functions attached to one polynomial, see \cite[Proposition 4.6]{Hoo1}.

\begin{proposition}\label{cotasup}
	Let $f$, $g$ be non-zero polynomials and the character $\chi$  satisfy the conditions of Theorem \ref{Theorem2}.  Then for every pole  $s$ of $Z(s,\chi,\frac{f}{g})$ with $Re(s)<0$  one has  $Re(s)\leq -1/t_{0}$ if $\chi\neq\chi_{triv}$; otherwise  $Re(s)=-1$ or $Re(s)\leq -1/t_{0}$. 
\end{proposition}
\begin{proof}
The result follows from Propositions  \ref{PropPoles}, \ref{Prop3}, and the following inequality
	\begin{equation*}
-\frac{1}{t_{0}}\geq -\frac{\sigma(\boldsymbol{w})}{d(\boldsymbol{w},\Gamma(g))+d(\boldsymbol{w},\Gamma(f))}\geq -\frac{\sigma(\boldsymbol{w})}{d(\boldsymbol{w},\Gamma(f))-d(\boldsymbol{w},\Gamma(g))},\quad \boldsymbol{w}\in T_{-} .
\end{equation*}
	\end{proof}

 \begin{proposition}\label{cotainf}
 	Let $f$, $g$ be non-zero polynomials and the character $\chi$  satisfy the conditions of Theorem \ref{Theorem2}. Then for every pole  $s$ of $Z\left(s,\chi,\frac{f}{g}\right)$ with $Re(s)>0$  one has  $Re(s)\geq 1/t_{0}$ if $\chi\neq\chi_{triv}$; otherwise $Re(s)=1$ or $Re(s)\geq 1/t_{0}$.
 \end{proposition}
\begin{proof}
The result follows by using Propositions \ref{PropPoles}, \ref{Prop3} and the following  inequality
 \begin{equation*} \frac{\sigma(\boldsymbol{w})}{d(\boldsymbol{w},\Gamma(g))-d(\boldsymbol{w},\Gamma(f))}\geq \frac{\sigma(\boldsymbol{w})}{d(\boldsymbol{w},\Gamma(g))+d(\boldsymbol{w},\Gamma(f))}\geq \frac{1}{t_{0}}, \quad \boldsymbol{w}\in T_{+}.\end{equation*}
	\end{proof}

\begin{remark}
			Notice that if $d(\boldsymbol{w},\Gamma(f))\neq 0$ and $d(\boldsymbol{w},\Gamma(g))\neq 0$ for all  $\boldsymbol{w}\in\mathcal{D}(t_{0})$, we can conclude that $\pm \frac{1}{t_{0}}$ is not a pole of $Z\left(s,\chi, \frac{f}{g}\right)$.
In effect, if $\pm \frac{1}{t_{0}}$ is a pole of $Z\left(s,\chi, \frac{f}{g}\right)$ then it comes neccesarily from $S_{\Delta}$ for some $\Delta$ in the polyhedral subdivion. Hence if we assume that $\frac{1}{t_{0}}$ is a pole, then by Propositions \ref{PropPoles} and \ref{cotainf} , there exists $\boldsymbol{w}_{i}\in T_{+}$ such that 
	\begin{equation}\label{nonpole}
	\frac{1}{t_{0}}=\frac{\sigma(\boldsymbol{w}_{i})}{d(\boldsymbol{w}_{i},\Gamma(g))-d(\boldsymbol{w}_{i},\Gamma(f))}\geq \frac{\sigma(\boldsymbol{w}_{i})}{d(\boldsymbol{w}_{i},\Gamma(g))+d(\boldsymbol{w}_{i},\Gamma(f))}\geq \frac{1}{t_{0}} .
	\end{equation}
	Then by Proposition \ref{Prop3}, $\boldsymbol{w}_{i}\in \mathcal{D}(t_{0})$ and the inequation (\ref{nonpole}) becomes to an equation, and this implies that  $d(\boldsymbol{w}_{i},\Gamma(f))=0$, which contradicts that $d(\boldsymbol{w}_{i},\Gamma(f))\neq 0$ for any  $\boldsymbol{w}_{i}\in \mathcal{D}(t_{0})$. Similarly,by using Proposition \ref{cotasup} we can show that $-\frac{1}{t_{0}}$ is not a pole of $Z\left(s,\chi, \frac{f}{g}\right)$.
\end{remark}

By this Remark,  it remains that $\pm \frac{1}{t_{0}}$ are possible poles  of $Z\left(s,\chi, \frac{f}{g}\right)$ only for the cases given in Theorems  \ref{Diagonal1} and \ref{Diagonal2}.
 
\begin{theorem}\label{Diagonal1}  Let $f$, $g$ be non-zero polynomials and the character $\chi$  satisfy the conditions of Theorem \ref{Theorem2}. If   $(t_{0},\ldots,t_{0})\in F(\boldsymbol{w},\Gamma(f))$ for some $\boldsymbol{w}\in\mathcal{D}(t_{0})$, it means that   $(t_{0},\ldots,t_{0})$ is the intersection point of the diagonal with the boundary of the Newton polyhedron $\Gamma(f)$. Then   $\beta = -1/t_{0}$ of expected order $\rho$ if $t_{0}\neq1$ and $\chi=\chi_{triv}$ or $\chi\neq \chi_{triv}$; otherwise  $\rho+1$ if  $t_{0}=1$ and $\chi=\chi_{triv}$. Moreover,  if $\chi=\chi_{triv}$ and $t_{0}>1$, then $-1/t_{0}$ is a pole of $Z(s,\chi,\frac{f}{g})$ of order $\rho$.
\end{theorem}
\begin{proof}
	Notice that the condition $(t_{0},\ldots,t_{0})\in F(\boldsymbol{w},\Gamma(f))$ implies $d(\boldsymbol{w},\Gamma(g))= 0$ and $d(\boldsymbol{w},\Gamma(f))\neq 0$, see Corollary \ref{Cor1}. Now, by using that  $\boldsymbol{w}\in\mathcal{D}(t_{0})$ and Proposition \ref{Prop3}, it follows that  \[-\frac{1}{t_{0}}=-\frac{\sigma(\boldsymbol{w})}{d(\boldsymbol{w},\Gamma(f))}.\] Hence by definition of the candidate poles  given in Proposition \ref{PropPoles}, $-1/t_{0}$ is a posible real negative part of a pole, thus  $w\in \mathcal{D}_{-}(t_{0}).$ Furthermore, the definition of $\beta$ and Proposition \ref{cotasup} implies that $\beta=-1/t_{0}$. Now, if $-1/t_{0}=-1$ when $\chi=\chi_{triv}$, and if $N_{\Delta,\{f\}}\neq 0$ or $N_{\Delta,\{f,g\}}\neq 0$ in the explicit form of $Z(s,\chi,\frac{f}{g})$, then the expected order is $\rho+1$; otherwise, the expected order is $\rho$.  The second part is exactly the result stablished in \cite[Theorem 3]{B-Z-0} by using that $\beta=-1/t_{0}$.
\end{proof}
\begin{theorem}\label{Diagonal2}
 Let $f$, $g$ be non-zero polynomials and the character $\chi$  satisfy the conditions of Theorem \ref{Theorem2}. Assume that $(t_{0},\ldots,t_{0})\in F(\boldsymbol{w},\Gamma(g))$ for  some $\boldsymbol{w}\in\mathcal{D}(t_{0})$, it means that   $(t_{0},\ldots,t_{0})$ is the intersection point of the diagonal with the boundary of the Newton polyhedron $\Gamma(g)$. Then $\alpha=1/t_{0}$  of expected order $\kappa$ if $t_{0}\neq1$ and $\chi=\chi_{triv}$ or $\chi\neq \chi_{triv}$; otherwise  $\kappa+1$ if  $t_{0}=1$ and $\chi=\chi_{triv}$. Moreover if $\chi=\chi_{triv}$ and  $t_{0}>1$, then $1/t_{0}$ is a pole of $Z\left(s,\frac{f}{g}\right)$ of order $\kappa$.
	\end{theorem}
\begin{proof}
	A similar argument given in Theorem \ref{Diagonal1} shows the first part of Theorem \ref{Diagonal2}. The second part follows from \cite[Theorem 4]{B-Z-0}, because $\alpha=1/t_{0}$, and $t_{0}>1$ implies $1/t_{0}<1$.
\end{proof}

\begin{example}
	\label{Example1}  Let $f(x,y)=x^{2}-y$,
	$g(x,y)=x^{2}y$ polynomials in ${\mathcal{O}}_{K}[x,y]$. 
	This example is given in Example 1 in  \cite{B-Z-0} and in the same paper, in Example 4, we computed the local zeta function attached to $f/g$ and we showed that its poles have real parts belonging to $\left\{  -1,1/2,1,3/2\right\}$. 
Notice that the intersection point of the diagonal with $\Gamma\left(\frac{f}{g}\right)$ is the point $(2,2)$. Thus, $t_{0}=2$ and  $\mathcal{D}(t_{0})=\{(1,0), (1,2)\}$. Furthermore, the hypothesis of Theorem \ref{Diagonal2} holds: $d((1,0),\Gamma(f))=0$, the intersection point $(2,2)\in F((1,0),\Gamma(g))$, and $t_{0}>1$. Hence, Theorem \ref{Diagonal2} implies that $1/2$ is a pole of $Z\left(s,\frac{f}{g}\right)$ of order $1$.
\end{example}
  
\section{ Positive and negative real poles of  $Z\left(s,\frac{f}{g}\right)$}\label{Sect_7}
In this section we determine some conditions under which $Z\left(s,\frac{f}{g}\right)$ has  a real pole when $f$, $g$ satisfy the conditions of Theorem  \ref{Theorem2}. Furthermore, we obtain information about the largest negative real pole and the smallest positive real pole and they orders, see Theorems \ref{negativepole} and \ref{positivepole}. Also, we notice that if the conditions given in Theorems \ref{Diagonal1} or \ref{Diagonal2} are satisfied then $Z\left(s,\frac{f}{g}\right)$ has always a real pole determined by the intersection point of the diagonal with the boundary of the Newton polyhedron $\Gamma\left(\frac{f}{g}\right)$: $-1/t_{0}$ or $1/t_{0}$, respectively.

\subsection{The largest negative  real pole of $Z\left(s,\frac{f}{g}\right)$}
We remark that if the conditions of Theorem \ref{Diagonal1} are satisfied then   $\beta=-\frac{1}{t_{0}}$, and the results given in Theorem \ref{negativepole} are similar to the ones given in the classical case of local zeta functions attached to one polynomial, see \cite[Theorem 4.10]{Hoo1}.

\begin{theorem}\label{negativepole}
\noindent (1) Assume that  $T_{-}\neq \varnothing$. Then the following hold:
		
		\noindent (1.a) Suppose that $\beta >-1$. Then $\beta$ is the largest real negative pole of $Z\left(s,\frac{f}{g}\right)$ with order $\rho$.
		
	\noindent (1.b) Suppose that  $\beta <-1$. Then if there  exists a cone $\Delta\in \mathcal{F}\left(\frac{f}{g}\right)$ such that $N_{\Delta,\{f\}}\neq 0$ or $N_{\Delta,\{f,g\}}\neq 0$  then $-1$ is the largest real negative pole of $Z\left(s,\frac{f}{g}\right)$ and its order will be 1. Otherwise, $\beta$ will be the largest real negative pole of $Z\left(s,\frac{f}{g}\right)$ and its order will be $\rho$.
		
		\noindent (1.c) if $\beta =-1$, then $\beta$ will be the  largest real negative pole of $Z\left(s,\frac{f}{g}\right)$. If there is a cone $\Delta\in \mathcal{M}_{\rho}\left(\beta\right)$ such that $N_{\Delta,\{f\}}\neq 0$ or $N_{\Delta,\{f,g\}}\neq 0$  then its order will be $\rho+1$. Otherwise, its order will be $\rho$.
		
		\noindent (2)	If $T_{-}=\varnothing$ and $N_{\Delta,\{f\}}\neq \varnothing$ or $N_{\Delta,\{f,g\}}\neq \varnothing$, then $-1$ is the unique real negative pole of $Z\left(s,\frac{f}{g}\right)$ and its  order is $1$.
\end{theorem}
\begin{proof} First we proof Part (1).
	Case (1.a) is given in \cite[Theorem 3]{B-Z-0}. For case (1.b) we recall that 
	
\[
Z\left(s,\frac{f}{g}\right)=\sum_{\text{$\Delta\in\mathcal{F}(\frac{f}{g})\cup$%
	}\left\{  \mathbf{0}\right\}  }L_{\Delta}\left(s,\frac{f}{g}\right)S_{\Delta}(s),
\]
 $\Delta\in\mathcal{F}\left(\frac{f}{g}\right)\cup$$\left\{  \mathbf{0}\right\}$, 
\begin{align*}
L_{\Delta}\left(s,\frac{f}{g}\right)&=
q^{-n}\left[  (q-1)^{n}-N_{\Delta,\left\{
	f\right\}  }\frac{1-q^{-s}}{1-q^{-1-s}}-N_{\Delta,\left\{  g\right\}  }%
\frac{1-q^{s}}{1-q^{-1+s}}\right. \\
&\left. -N_{\Delta,\left\{  f,g\right\}  }\frac{(1-q^{-s})(1-q^{s}%
	)}{q(1-q^{-1-s})(1-q^{-1+s})}\right]
\end{align*}

where $S_{\boldsymbol{0}}(s)=1$, and if
$\Delta\in\mathcal{F}\left(\frac{f}{g}\right)$ is a cone strictly positively generated by
linearly independent vectors $\boldsymbol{w}_{1},\ldots,\boldsymbol{w}_{l}%
\in\mathcal{D}(\frac{f}{g})$, then
\[
S_{\Delta}(s)=\frac{\sum_{\boldsymbol{t}}q^{-\sigma(\boldsymbol{t}
		)-(d(\boldsymbol{t},\Gamma(f))-d(\boldsymbol{t},\Gamma(g))){s}}}{\prod
	_{i=1}^{l}(1-q^{-\sigma(\boldsymbol{w}_{i})-(d(\boldsymbol{w}_{i}%
		,\Gamma(f))-d(\boldsymbol{w}_{i},\Gamma(g))){s}})},
\]
In order to proof the first part of case (1.b) we assume that there exist a cone $\Delta_{0}\in \mathcal{F}\left(\frac{f}{g}\right)$ such that $N_{\Delta_{0},\{f\}}\neq 0$ or $N_{\Delta_{0},\{f,g\}}\neq 0$. 

To prove that $-1$ is a pole of $Z\left(s,\frac{f}{g}\right)$ of order $1$, it is sufficient to show that
\[Res\left(\Delta,-1\right):=\lim_{s\rightarrow-1
}(1-q^{-1-s})L_{\Delta}\left(s,\frac{f}{g}\right)S_{\Delta}(s)\geq0\] for every cone $\Delta\in\mathcal{F}(\frac{f}{g})$ and  $Res\left(
\Delta_{0},-1\right) \allowbreak>0$.

Notice that 
\begin{equation}
\lim_{s\rightarrow -1}S_{\Delta}\left(  s\right)  >0
\label{calculo de residuo2}%
\end{equation}
for all cones $\Delta\in\mathcal{F}(\frac{f}{g})\cup\left\{  \mathbf{0}%
\right\}  $. Inequality (\ref{calculo de residuo2}) follows from 
\[
\lim_{s\rightarrow-1}\sum_{\boldsymbol{t}}q^{-\sigma(\boldsymbol{t}%
	)-(d(\boldsymbol{t},\Gamma(f))-d(\boldsymbol{t},\Gamma(g))){s}}>0,
\]
and
\[
1-q^{-\sigma(\boldsymbol{w}_{i})-(d(\boldsymbol{w}_{i},\Gamma
	(f))-d(\boldsymbol{w}_{i},\Gamma(g)))(-1)}>0
\]
because $\beta <-1$ implies $-\sigma(\boldsymbol{w}_{i})-(d(\boldsymbol{w}_{i},\Gamma
(f))-d(\boldsymbol{w}_{i},\Gamma(g)))(-1)<0$ for any $\boldsymbol{w}%
_{i}\in T_{+}\cup T_{-}$. From these observations, we
have%
\begin{gather*}
\lim_{s\rightarrow-1}\frac{\sum_{\boldsymbol{t}}q^{-\sigma(\boldsymbol{t}
		)-(d(\boldsymbol{t},\Gamma(f))-d(\boldsymbol{t},\Gamma(g))){s}}}{\prod
	_{i=1}^{l}(1-q^{-\sigma(\boldsymbol{w}_{i})-(d(\boldsymbol{w}_{i}%
		,\Gamma(f))-d(\boldsymbol{w}_{i},\Gamma(g))){s}})}>0.
\end{gather*}
Now, we prove that
\begin{equation}\label{ecuacion2}
 \lim_{s\rightarrow-1}(1-q^{-1-s})L_{\Delta}\left(s,\frac{f}{g}\right)\geq0.
\end{equation}
By definition of $L_{\Delta}\left(s,\frac{f}{g}\right)$, for each cone such that $N_{\Delta,\{f\}}=0=N_{\Delta,\{f,g\}}$, we have \[ \lim_{s\rightarrow-1}(1-q^{-1-s})L_{\Delta}\left(s,\frac{f}{g}\right)=0\]
On the other hand, we assume that $N_{\Delta,\{f\}}\neq 0$ and $N_{\Delta,\{f,g\}}\neq 0$. The other cases when one of these is zero are treated in a similar way.  Hence
\begin{align*}
 \lim_{s\rightarrow-1}(1-q^{-1-s})L_{\Delta}\left( s,\frac{f}{g}\right)  &>q^{-n}\left( (q-1)^{n}+N_{\Delta
	,\left\{  f\right\}  }(q-1)\right.\\
&\left.-N_{\Delta,\left\{  g\right\}  }+N_{\Delta,\left\{
	f,g\right\}  }\left(1-q^{-1}\right)\right)>0.
\end{align*}
In particular, $Res\left(
\Delta_{0},-1\right) \allowbreak>0$. This shows that $-1$ is a pole of order $1$.

Now,  the second part of case (1.b)  and $\beta<0$ implies that 

\[L_{\Delta}\left(\beta,\frac{f}{g}\right)=q^{-n}\left[ (q-1)^{n}-N_{\Delta,\left\{  g\right\}  }\frac{1-q^{\beta}}{1-q^{-1+\beta}}\right] 
>q^{-n}\left[ (q-1)^{n}-N_{\Delta,\left\{  g\right\}  }\right]>0. \] Thus, it is sufficient to prove that 
$\lim_{s\rightarrow\beta
}(1-q^{s-\beta})^{\rho}S_{\Delta}(s)\geq0$\textbf{
}for every cone $\Delta\in\mathcal{F}(\frac{f}{g})$ and that there exists
a cone $\Delta_{0}\in\mathcal{M}_{\rho}(\beta)$ such that $Res\left(
\Delta_{0},\beta\right)  \allowbreak>0$.

We show that\ for at least one cone $\Delta_{0}$ in $\mathcal{M}_{\rho}%
(\beta)$, $Res\left(  \Delta_{0},\beta\right)  \allowbreak>0$, because for any
cone $\Delta\notin\mathcal{M}_{\rho}(\beta)$, $Res\left(  \Delta,\beta\right)
=0$. This last assertion can be verified by using the argument that we give
for the cones in $\mathcal{M}_{\rho}(\beta)$. We first note that there exists
at least one cone $\Delta_{0}$ in $\mathcal{M}_{\rho}(\beta)$. Let
$\boldsymbol{w}_{1},\ldots,\boldsymbol{w}_{\rho},\allowbreak\boldsymbol{w}%
_{\rho+1},\ldots,\boldsymbol{w}_{l}$ its generators with $\boldsymbol{w}%
_{i}\in\mathcal{P}(\beta)\Leftrightarrow1\leq i\leq\rho$.

We notice that
\[
\lim_{s\rightarrow\beta}\sum_{\boldsymbol{t}}q^{-\sigma(\boldsymbol{t}%
	)-(d(\boldsymbol{t},\Gamma(f))-d(\boldsymbol{t},\Gamma(g))){s}}>0.
\]
Hence in order to show that $Res\left(  \Delta_{0},\beta\right)
\allowbreak>0$, it is sufficient to show that
\[
\lim_{s\rightarrow\beta}\frac{(1-q^{s-\beta})^{\rho}}{\prod_{i=1}%
	^{l}(1-q^{-\sigma(\boldsymbol{w}_{i})-(d(\boldsymbol{w}_{i},\Gamma
		(f))-d(\boldsymbol{w}_{i},\Gamma(g))){s}})}>0.
\]
Now, notice there are positive integer constants $c_{i}$ such that
\begin{gather*}
\prod_{i=1}^{\rho}(1-q^{-\sigma(\boldsymbol{w}_{i})-(d(\boldsymbol{w}%
	_{i},\Gamma(f))-d(\boldsymbol{w}_{i},\Gamma(g))){s}})=\prod_{i=1}^{\rho
}(1-q^{\left(  {s}-\beta\right)  c_{i}})\\
=(1-q^{s-\beta})^{\rho}\prod_{i=1}^{\rho}\text{ }\prod\limits_{\varsigma
	^{c_{i}}=1,\varsigma\neq1}\left(  1-\varsigma q^{s-\beta}\right)  .
\end{gather*}
In addition, for $i=\rho+1,\ldots,l$,
\[
1-q^{-\sigma(\boldsymbol{w}_{i})-(d(\boldsymbol{w}_{i},\Gamma
	(f))-d(\boldsymbol{w}_{i},\Gamma(g))){\beta}}>0
\]
because $-\sigma(\boldsymbol{w}_{i})-(d(\boldsymbol{w}_{i},\Gamma
(f))-d(\boldsymbol{w}_{i},\Gamma(g))){\beta<0}$ for any $\boldsymbol{w}%
_{i}\in T_{+}\cup T_{-}$ with $i=\rho+1,\ldots,l$. From these observations, we
have%
\begin{gather*}
\lim_{s\rightarrow\beta}\frac{(1-q^{s-\beta})^{\rho}}{\prod_{i=1}%
	^{l}(1-q^{-\sigma(\boldsymbol{w}_{i})-(d(\boldsymbol{w}_{i},\Gamma
		(f))-d(\boldsymbol{w}_{i},\Gamma(g))){s}})}=\\
\lim_{s\rightarrow\beta}\frac{(1-q^{s-\beta})^{\rho}}{(1-q^{s-\beta})^{\rho
	}\prod_{i=1}^{\rho}\text{ }\prod\limits_{\varsigma^{c_{i}}=1,\varsigma\neq
		1}\left(  1-\varsigma q^{s-\beta}\right)  }\times\\
\lim_{s\rightarrow\beta}\frac{1}{\prod_{i=\rho+1}^{l}(1-q^{-\sigma
		(\boldsymbol{w}_{i})-(d(\boldsymbol{w}_{i},\Gamma(f))-d(\boldsymbol{w}%
		_{i},\Gamma(g))){s}})}>0.
\end{gather*}
	The case (1.c) and the proof of part (2) are similar to the case (1.b) and the first case of (1.b), respectively.
	\end{proof}

\subsection{The smallest positive  real pole of $Z\left(s,\frac{f}{g}\right)$}
The following theorem is the `posivite' counterpart of the classical theory of local zeta functions attached to one polynomial.
\begin{theorem}\label{positivepole}
\noindent (1') Assume that  $T_{+}\neq \varnothing$. Then the following hold:
	
	\noindent (1.a') Suppose that $\alpha <1$. Then $\alpha$ is the smallest real positive pole of $Z\left(s,\frac{f}{g}\right)$ with order $\kappa$. 
	
	\noindent (1.b') Suppose that  $\alpha >1$. Then if there  exists a cone $\Delta\in \mathcal{F}\left(\frac{f}{g}\right)$ such that $N_{\Delta,\{g\}}\neq 0$ or $N_{\Delta,\{f,g\}}\neq 0$  then $1$ is the smallest real positive pole of $Z\left(s,\frac{f}{g}\right)$ of order 1. Otherwise, $\alpha$ will be the smallest real positive pole of $Z\left(s,\frac{f}{g}\right)$ and its order will be $\kappa$.
	
	\noindent (1.c') if $\alpha =1$, then $\alpha$ will be the  smallest real positive pole of $Z\left(s,\frac{f}{g}\right)$. If there is a cone $\Delta\in \mathcal{F}\left(\frac{f}{g}\right)$ such that $N_{\Delta,\{g\}}\neq 0$ or $N_{\Delta,\{f,g\}}\neq 0$  then its order will be $\kappa+1$. Otherwise, its order will be $\kappa$.
	
	\noindent (2')	If $T_{+}=\varnothing$ and $N_{\Delta,\{g\}}\neq \varnothing$ or $N_{\Delta,\{f,g\}}\neq \varnothing$, then $1$ is the unique real positive pole of $Z\left(s,\frac{f}{g}\right)$ and its  order is $1$.
\end{theorem}
\begin{proof}
	The case (1.a') is well known, see \cite[Theorem 4]{B-Z-0}. The other cases are similar to cases (1.b), (1.c) and (2) in Theorem \ref{negativepole}.
	\end{proof}

\begin{example}
	We consider a simple case $f(x,y)=x$,
	$g(x,y)=y$ in ${\mathcal{O}}_{K}[x,y]$. We notice that $f/g$ is non-degerate  over $\mathbb{F}_{q}$ with respect to $\Gamma\left(\frac{f}{g}\right)$ due to  the subset defined in (\ref{Condition_1}) is empty for any $\boldsymbol{k}\in\mathbb{R}^{n}_{+}$. In this case, $t_{0}=1$,  $\mathcal{D}(t_{0})=\{(1,0),(0,1)\}$, and $(1,1)\in F((1,0),\Gamma(f))$ and $(1,1)\in F((0,1),\Gamma(g))$,  then by Theorems \ref{Diagonal1}, \ref{Diagonal2}, \ref{negativepole} and \ref{positivepole},  $\pm 1$ are the unique poles of $Z\left(s,\frac{f}{g}\right)$ and they orders are $1$.
\end{example}

\begin{remark}\label{Char} It follows by definition of $d(\boldsymbol{k},\Gamma)$ that
	\begin{itemize}
		\item[i)] if $\Gamma(g)\subset \Gamma(f)$, then $T_{-}=\varnothing$, then -1 could be the only negative pole of $Z\left(s,\frac{f}{g}\right)$, see (2) in Theorem \ref{negativepole}, and its smallest positive real pole is given by Theorem \ref{positivepole},
		\item[ii)] if $\Gamma(f)\subset \Gamma(g)$, then $T_{+}=\varnothing$, then 1 could be the only positive pole of $Z\left(s,\frac{f}{g}\right)$, see (2') in Theorem \ref{positivepole}, and its largest negative real pole is given by Theorem \ref{negativepole},
		\item[iii)] if $\Gamma(f)=\Gamma(g)$, then $T_{+}=T_{-}=\varnothing$. Thus the only real poles of $Z\left(s,\frac{f}{g}\right)$ could be the trivial ones -1 or 1. 
		\item[iv)] if $\Gamma(f)=\Gamma(g)$ and $N_{\Delta,\{f\}}=N_{\Delta,\{g\}}=N_{\Delta,\{f,g\}}=0$ for any $\Delta\in\mathcal{F}\left(\frac{f}{g}\right)\cup\{\boldsymbol{0}\}$. Then  $Z\left(s,\frac{f}{g}\right)$ does not have poles. 
	\end{itemize}	
	
\end{remark}

\begin{example}
	\label{Example2}  Let $f(x,y)=x^{2}+y^{2}$,
	$g(x,y)=x^{4}+y^{4}$ be two polynomials in $\mathbb{Q}_{p}[x,y]$, and we assume that $-1$ is not a square in $\mathbb{Q}_{p}$, i. e., $p\equiv3\bmod 4$.  A simplicial
	polyhedral subdivision $\mathcal{F}\left(\frac{f}{g}\right)$ subordinate to $\Gamma(\frac{f}{g})$ is given by
\end{example}
\begin{center}%
	\begin{table}[h!]
		\begin{tabular}
			[c]{|l|l|l|}\hline
			Cone & $f_{b(\Delta)}$ & $g_{b(\Delta)}$\\\hline
			$\Delta_{1}:=(1,0){\mathbb{R}}_{>0}$ & $y^{2}$ & $y^{4}$\\\hline
			$\Delta_{2}:=(1,0){\mathbb{R}}_{>0}+(1,1){\mathbb{R}}_{>0}$ & $y^{2}$ & $y^{4}$\\\hline
			$\Delta_{3}:=(1,1){\mathbb{R}}_{>0}$ & $x^{2}+y^{2}$ & $x^{4}+y^{4}$\\\hline
			$\Delta_{4}:=(1,1){\mathbb{R}}_{>0}+(0,1){\mathbb{R}}_{>0}$ & $x^{2}$ &
			$x^{4}$\\\hline
			$\Delta_{5}:=(0,1){\mathbb{R}}_{>0}$ & $x^{2}$ & $x^{4}$\\\hline
		\end{tabular}\caption{Cones and face functions}
	\end{table}
\end{center}
\noindent where ${\mathbb{R}}_{>0}:={\mathbb{R}}_{+}\smallsetminus\left\{
\boldsymbol{0}\right\}  $. Notice that for every $\boldsymbol{k}\in
\mathbb{R}_{+}^{n}\smallsetminus(\left\{  \boldsymbol{0}\right\}  \cup
\Delta_{3})$ and every non-empty subset $I\subseteq\left\{  1,2\right\}  $,
the subset defined in (\ref{Condition_1}) is empty, thus (\ref{Condition_2})
is always satisfied. In the case $\boldsymbol{k}=\boldsymbol{0}$ and
$\boldsymbol{k}\in\Delta_{3}$, $f_{\boldsymbol{k}}=x^{2}+y^{2}$,
$g_{\boldsymbol{k}}=x^{4}+y^{4}$, the conditions (\ref{Condition_1}%
)-(\ref{Condition_2}) are also verified because of $-1$ is not a square in $\mathbb{Q}_{p}$.  Hence $f/g$ is
non-degenerate over $\mathbb{F}_{q}$ with respect to $\Gamma\left(
\frac{f}{g}\right) $ and $L_{\Delta}=q^{-2}(q-1)^{2}$ for any $\Delta\in\mathcal{F}\left(\frac{f}{g}\right)\cup\{\boldsymbol{0}\}$.  Notice that $\Gamma(g)\subset\Gamma(f)$, then Remark \ref{Char} i) implies that $\alpha=1$ is the smallest real positive pole of $Z\left(s,\frac{f}{g}\right)$ of order $1$. Moreover, it is the unique pole of the local zeta function, because of $d((1,0),\Gamma(f))=d((0,1),\Gamma(f))=0$ and $d((1,0),\Gamma(g))=d((0,1),\Gamma(g))=0$.

\begin{example}
If in the Example \ref{Example2} we consider $f(x,y)=x^{4}+y^{4}$ and $g(x,y)=x^{2}+y^{2}$, then  $\Gamma(f)\subset \Gamma(g)$, then as above we can show  that $-1$ is the only pole of $Z\left(s,\frac{f}{g}\right)$ and its order is $1$. 
\end{example}

\end{document}